\documentclass[11pt,a4paper,reqno]{amsart}

\usepackage{amsmath}
\usepackage{amsthm}
\usepackage{amssymb}
\usepackage{graphics}
\usepackage{graphicx}
\usepackage{colortab} 
\usepackage{enumerate}

\usepackage{colortab}

\usepackage[colorlinks=true,linkcolor=blue,citecolor=blue]{hyperref}
\usepackage{prettyref}

\DeclareMathOperator{\supp}{supp}
\theoremstyle{plain}
\newtheorem{theorem}{\bf Theorem}[section]
\newtheorem{proposition}[theorem]{\bf Proposition}
\newtheorem{lemma}[theorem]{\bf Lemma}
\newtheorem{corollary}[theorem]{\bf Corollary}

\theoremstyle{remark}
\newtheorem{remark}[theorem]{\bf Remark}

\newcommand{\ud}[0]{\,\mathrm{d}}

\newcommand{\esssup}[0]{\operatornamewithlimits{ess\,sup}}
\newcommand{\essinf}[0]{\operatornamewithlimits{ess\,inf}}

\def\supp{\operatorname{supp}}

\def\esssup{\operatornamewithlimits{ess\,sup}}
\def\essinf{\operatornamewithlimits{ess\,inf}}

\def\R{\mathbb R}

\begin{document}
	\title[Restricted weak type extrapolation for $A_{\vec{p}}$ weights]{Endpoint multilinear restricted weak type extrapolation theorem}

\author[K. Li]{Kangwei Li}
\address[K. Li]{Center for Applied Mathematics, Tianjin University, Weijin Road 92, 300072 Tianjin, China}
\email{kli@tju.edu.cn}

\author[T. Luque ]{Teresa Luque}
\address[T. Luque]{Departamento de Análisis Matemático y Matemática Aplicada\\ Universidad Complutense (Spain)}
\email{t.luque@ucm.es}

\author[S. Ombrosi]{Sheldy Ombrosi}
\address[S. Ombrosi]{Departamento de Análisis Matemático y Matemática Aplicada\\ Universidad Complutense (Spain)\&
Instituto de Matemática. Universidad Nacional del Sur - CONICET Argentina}
\email{sombrosi@ucm.es}

\thanks{The first author author was supported by the National Natural Science Foundation of China through project numbers 12222114 and 12001400. 
The second and third authors were supported by Spanish Ministerio de Ciencia e Innovaci\'on grant PID2020-113048GB-I00.}
	
	\keywords{Weighted inequalities; multilinear Muckenhoupt weights; Rubio de Francia extrapolation;  restricted weak type.}

	\maketitle
	
	\begin{abstract}
		In this paper we present a generalization in the context of multilinear Muckenhoupt classes of the endpoint extrapolation theorem on restricted weights due to Carro, Grafakos and Soria \cite{cgs:cgs} . Moreover, our main result is obtained on limited ranges of boundedness and to this aim we introduce a new limited range, off-diagonal extrapolation theorem in the context of restricted weights. In addition, as one of the applications, we prove endpoint estimates of certain bi-sublinear maximal functions associated with the study of return time theorems in ergodic theory. 
	\end{abstract}
	
	\section{Introduction}
	
The Rubio de Francia  extrapolation theorem \cite{RdF:RdF} raises one of the most striking issues concerning weighted inequalities. This theorem posits that an estimate for an operator $T$ on $L^{p_0}(w)$ for a single $p_0$ and all $A_{p_0}$  weights $w$ implies a similar $L^p(w)$ estimate for all $p$ in $(1,\infty)$ and every $w\in A_p$. After this result was first proved, many other versions and improvements have appeared in the literature. A more detailed background can be found in \cite{CU-M-P}. In particular, the operator itself does not play any role in the theorem and the statement can be given in terms of families of pairs of nonnegative measurable functions (see \cite{Cup:Cup}). The extrapolation can also be generalized for the off-diagonal case in which the inequalities are from $L^p(w^p)$ to $L^q(w^q )$ for possibly different values for the exponents $p$ and $q$ and the weight $w$ in the appropriate classes $A_{p,q}$ (see \cite{hms:hms,lmpt:lmpt, duo:duo} ). Although for operators that are unbounded outside a range of the form $(p^{-},p^{+})$ with $1 < p^{-} <p^{+} <\infty$ classical extrapolation does not work, it has been extended to this scenario in \cite{am:am,duo:duo}.
	
	In the multi-variable setting, there are some Rubio de Francia's extrapolation results. In \cite{gm:gm} it was shown that if $T$ is bounded from $L^{p_1}(w_1)\dots \times L^{p_m}(w_m)$ into $L^p(w_1^{\frac{p}{p_1}}\dots w_m^{\frac{p}{p_m}})$ for  $w_i\in A_{p_i}$ and for some fixed exponents $1 <p_1,\dots,p_m<\infty$ with the relation given by
	\begin{equation}\label{re:exponents}
		\frac{1}{p}=\frac{1}{p_1}+\dots +\frac{1}{p_m},
	\end{equation}
	then the same holds for all possible values of $p_j$, $p_j>1$. This result is also obtained in \cite{duo:duo} using by iteration an off-diagonal extrapolation argument. These first results treat  each variable separately with its own Muckenhoupt class of weights and do not really use  the multivariable nature of the problem. In \cite{lmo:lmo} is presented an extrapolation argument using the multilinear Muckenhoupt classes of weights introduced in \cite{loptt:loptt}. Namely, given  $\vec{p}=(p_1,\dots, p_m)$  with the relation \eqref{re:exponents}, $\vec{w}=(w_1,\dots, w_m)$ is in the class $A_{\vec{p}}$ if
	\begin{equation}\label{Apvect}
		\sup_{Q}  \left( \frac{1}{|Q|} \int_{Q} w\right)^{\frac{1}{p}}  \prod_{i=1}^m \left( \frac{1}{|Q|} \int_{Q} w_i^{1-p'_i}\right)^{\frac{1}{p'_i}} <\infty,
	\end{equation}
	where $w=w_1^{\frac{p}{p_1}}\dots w_m^{\frac{p}{p_m}} $ and the exponents verify the relation \eqref{re:exponents}. These classes of weights characterize the boundedness of the sub-multilinear maximal function $\mathcal{M}$ defined by
	\[\mathcal M(f_1,\dots,f_m)(x)=\sup_{x\in Q} \prod_{i=1}^m \left(\frac{1}{|Q|}\int_Q|f_i(y)|\ud y \right).\]
It was also shown in \cite{loptt:loptt} that this is the appropriate class of weights for multilinear Calder\'on-Zygmund operators.

In \cite{lmo:lmo}, the authors exploited the clever off-diagonal argument presented in \cite{duo:duo} to prove multilinear extrapolation for the class $A_{\vec{p}}$. One of the key points there it was to understand that the multilinear condition (\ref{Apvect}) could be ``linearized"  by changing to a weighted measure space (see Lemma 3.2 in \cite{lmo:lmo}). Subsequently, a quantitative extrapolation result for these weights was obtained in \cite{n:n} via a generalization of the Rubio de Francia algorithm to the multilinear setting. The argument in \cite{n:n} allows also one to reach the upper endpoint. For more results concerning the upper endpoint extrapolation in the linear and multilinear context  see also \cite{lmmov, LOR, N-R}.
	
A weakness of Rubio de Francia's theory is that it does not allow to extrapolate up to the (lower) endpoint. More precisely, in the one-variable case it is not possible to obtain an estimate for $p=1$. In \cite{cgs:cgs}, the authors present a solution for this matter introducing a slightly bigger class than $A_{p_0}$, named $\widehat{A}_{p_0}$, for which given a restricted weak-type $(p_0,p_0)$ bounded operator $T:L^{p_0,1}(w)\to L^{p_0,\infty}(w) $, for some $p_0>1$ and every $w\in\widehat{A}_{p_0}$, then $T:L^{1,\frac{1}{p_0}}(w)\to L^{1,\infty}(w)$ for every  $w\in{A}_{1}$. The class $\widehat{A}_{p}$ is closely related to the restricted $A_p$ class, $A_p^{\mathcal{R}}$, which characterizes the restricted weak-type boundedness of the Hardy-Littlewood maximal function $M$ (see \cite{kt:kt} for this characterization).  
	
Extrapolation involving the classes $\widehat{A}_{p}$ and $A_p^{\mathcal{R}}$ is known as restricted weak type Rubio de Francia’s extrapolation and after the pioneering paper  \cite{cgs:cgs}, different applications have been developed around it. In particular, in the multi-variable context this argument allows to extrapolate down to $(1,\dots,1)$, considering hypothesis in each variable (weight) separately. See \cite{cr:cr} and \cite{rothesis}.
	
The purpose of this paper is to establish multi-variable restricted weak type extrapolation results to down to the lower endpoint for vector weights $\vec{w}=(w_1,\dots, w_m)$ associated with restricted weak type estimates. The first result in this direction can be found in \cite[Chapter 5, Theorem 5.2.2]{rothesis} where the $A_{\vec{p}}^{\mathcal{R}}$ class was introduced (see also \cite{P-R} and \cite{B-C-M}). We recall now the definition of theses classes.
	
Let $1\le p_i<\infty$ and $\frac{1}{p}=\frac{1}{p_1}+\dots+\frac{1}{p_m}$. We say $\vec w=(w_1, \cdots, w_m)\in A_{\vec p}^{\mathcal R}$ if 
\begin{equation}\label{eq:defr=1}
\sup_Q \left(\frac 1{|Q|} \int_Q\prod_{i=1}^mw_i^{\frac p{p_i}}\right)^{1/p}\prod_{i=1}^m \| \chi_Q w_i^{-1} \|_{L^{p_i', \infty}\big(\frac{w_i}{|Q|}\big)}<\infty,
\end{equation}where recall that when $p_i=1$ then 
\[
\| \chi_Q w_i^{-1} \|_{L^{p_i', \infty}\big(\frac{w_i}{|Q|}\big)}:= \esssup_Q w_i^{-1}.
\]
Observe that if $p_i=1$ for $i=1,\dots,m$ and therefore $p=\frac{1}{m}$  the class  $A_{\vec{\frac{1}{m}}}^{\mathcal R}$ actually agrees with the class $A_{\vec{\frac{1}{m}}}$  introduced in \cite{loptt:loptt}.
 
 It is proven (see \cite{P-R, rothesis}) $\vec{w} \in A_{\vec{p}}^\mathcal{R}$ if and only if $\mathcal{M}:L^{p_1,1}(w_1)\dots \times L^{p_m,1}(w_m)\to L^p(w_1^{\frac{p}{p_1}}\dots w_m^{\frac{p}{p_m}})$. Moreover, the recent preprint \cite{ro:ro} presents  the first restricted weak-type extrapolation scheme assuming multi-variable conditions (see \cite[Theorem 8.21]{ro:ro}). However, there the extrapolation is only in one component, the starting point is $(p_1,1,...,1)$ ($p_1>1$) to go down to $(1,...,1)$. Finally in the current paper we complete the picture and we can obtain the following result, which represents one of the main contributions of this work.
\begin{theorem}\label{thm:main}
Given measurable functions $f_1, \cdots, f_m$ and $g$. Suppose that for some exponents $1\le p_1, \cdots, p_m<\infty$ with $1/p=\sum_i 1/{p_i}$ and for all $\vec w=(w_1, \cdots, w_m)\in A_{\vec p}^{\mathcal R}$ we have 
\begin{equation}\label{eq:ApR}
\|g\|_{L^{p,\infty}(\prod_{i=1}^m w_i^{p/{p_i}})}\lesssim \prod_{i=1}^m \|f_i\|_{L^{p_i,1}(w_i)}.
\end{equation}
Then for all $(v_1, \cdots, v_m)\in A_{\vec{\frac{1}{m}}}$, 
\begin{equation}\label{eq:conclu}
\|g\|_{L^{1/m,\infty}(\prod_{i=1}^m v_i^{1/{m}})}\lesssim \prod_{i=1}^m \|f_i\|_{L^{1,1/{p_i}}(v_i)}.
\end{equation}
\end{theorem}

As it happens we are able to prove a more general theorem. In fact, a limited range version to reach the endpoint case is possible, and as a particular case, it gives us Theorem \ref{thm:main}. One of the interests of having a  limited range version of the previous theorem is associated with the fact that depending on the singularity of the operators, a higher local integrability of the intervening weights is necessary for their boundedness. See, for example, \cite{H-M}, \cite{C-D-O} or \cite{BM} where that fact is well reflected in the multilinear case.

To state this theorem, we need to introduce some notation. We start by defining the general restricted weights class $A_{\vec p, \vec r}^{\mathcal R}$. We say $\vec w\in A_{\vec p, \vec r}^{\mathcal R}$, if 
\begin{align*}
[\vec w]_{A_{\vec p, \vec r}}^{\mathcal R}:= \sup_Q \Big(\frac 1{|Q|}\int_Q w^{\frac{\delta_{m+1}}{p}}\Big)^{\frac 1{\delta_{m+1}}}\prod_{i=1}^m \| \chi_Q w_i^{-\frac 1{p_i}}\|_{L^{\delta_i, \infty}(\frac{\ud x}{|Q|})}<\infty,
\end{align*}
where $w=\prod\limits_{i=1}^m w_i^{p/{p_i}}$, $1\le r_1,\ldots, r_{m+1}<\infty$, $r_i\le p_i$ for all $1\le i\le m$, $p< r_{m+1}'$ and 
\begin{align*}
\frac 1p=\sum_{i=1}^m\frac 1{p_i},\quad \frac 1{\delta_i}:= \frac 1{r_i}-\frac 1{p_i}\,\,\forall \,1\le i\le m+1,\quad \frac 1{p_{m+1}}=1-\frac 1p.
\end{align*}
Here if $r_i=p_i$ we have the expression $$\| \chi_Q w_i^{-\frac 1{p_i}}\|_{L^{\delta_i, \infty}(\frac{\ud x}{|Q|})}=\esssup_{Q} w_i^{-\frac 1{p_i}}.$$ 
Note that our formulation is different from \eqref{eq:defr=1} when $r_i =1$ for every $i=1,\dots,m+1$. However, thanks to  Lemma \ref{red} below we see that these two formulations coincide (see Remark \ref{rem:dfnequiv} for details). 

Now with the previous notation we can state the main result of this paper. 

\begin{theorem}\label{thm:main-general}
Given measurable functions $f_1, \cdots, f_m$ and $g$ and some exponents $1\le r_i\le p_i<\infty$,  $\forall \,\, 1\le i \le m$, with $1/p=\sum_i 1/{p_i}>  1/{r_{m+1}'}$ for some $1\le r_{m+1}<\infty$. Suppose that for some constants $\alpha_i\in (0,p_i]$ and for all $\vec w=(w_1, \cdots, w_m)\in A_{\vec p, \vec r}^{\mathcal R}$ we have 
\begin{equation}\label{eq:ApR}
\|g\|_{L^{p,\infty}(\prod_{i=1}^m w_i^{p/{p_i}})}\lesssim \prod_{i=1}^m \|f_i\|_{L^{p_i,\alpha_i}(w_i)}.
\end{equation}
Then for all $(v_1, \cdots, v_m)\in A_{(r_1, \cdots, r_m), \vec r}$, 
\begin{equation}\label{eq:conclu}
\|g\|_{L^{\widetilde r,\infty}(\prod_{i=1}^m v_i^{\widetilde r/{r_i}})}\lesssim \prod_{i=1}^m \|f_i\|_{L^{r_i,\alpha_i r_i/{p_i}}(v_i)},
\end{equation}
where $\widetilde r$ is defined via $1/{\widetilde r}=\sum_{i=1}^m 1/{r_i}$. 
\end{theorem}
\begin{remark}
Observe  that if $\alpha_i=p_i$ for all $1\le i\le m$, then we genuinely end up with   a weak type estimate. 
\end{remark}

{\it{An idea of the proof of Theorem  \ref{thm:main-general} :}} As the multi-variable extrapolation presented in \cite{lmo:lmo} that follows the scheme of \cite{duo:duo}, the proof of this theorem is based on two main ingredients: an off-diagonal theorem in the restricted context and a structural result concerning the $A_{\vec{p}, \vec r}^{\mathcal R}$ classes. In fact, we believe that, the result that deals with the structure of the $A_{\vec{p}, \vec r}^{\mathcal{R}}$  (see Proposition \ref{impli-g} in the next section) represents one of the novel points of this paper even in the particular case of $r_i=1$ for $i=1,\dots,m+1$. It shows the connection between the multi-components of the vector weight in $A_{\vec{p}, \vec r}^{\mathcal R}$ classes. 
An analogous characterization for the $A_{\vec{p},\vec{r}}$ classes can be found in \cite[Lemma 3.2]{lmo:lmo}. In that case or in the equivalent presented in \cite[Theorem 3.6]{loptt:loptt}  (for $A_{\vec{p}}$) complete characterizations are achievable due to (among other things)  each weight $\sigma_i=w_i^{1-p_{i}'}\in A_{\infty}$. However, in our context, the $A_{\vec{p}}^{\mathcal R}$ classes generally do not exhibit such good property in the weights $\sigma_i$.  To overcome this obstacle and obtain a useful variant, a deeper understanding of weak norms is necessary. This understanding is partially facilitated by a Sawyer-type inequality, which can be found in, for example, \cite{CU-M-P0}.

The paper is organized as follows. In the next section we provide an off-diagonal extrapolation theorem. In Section 3 we obtain a partial characterization of the classes $A_{\vec{p}, \vec r}^{\mathcal{R}}$ and we give the proof of Theorem \ref{thm:main-general}. Finally, in the last section we present several applications.

	\section{A general off-diagonal extrapolation theorem}
This section is devoted to proving a general off-diagonal extrapolation theorem.  First, we need to introduce some notation. 

	Let $\mu$  be a doubling measure. For $1\le r <\infty$ we will say that $v\in \widehat{A}_r(\mu)$ if and only if there exist $u_1\in A_{1}(\mu)$ and a function $g$ with $M_{\mu}g(x)<\infty$ a.e $x$ such that $v:=u_1 M_{\mu}g^{1-r}$. Observe that the class $\widehat{A}_1(\mu)$ agrees with $A_{1}(\mu)$.
Moreover, if  $0<q<\infty$ and $1\le r <\infty$ we will say that  $v\in \widehat{A}_{r,q}(\mu)$ if and only if $v^q\in \widehat{A}_{1+\frac{q}{r'}}(\mu)$. 

We will need the following result.

\begin{lemma}[\cite{CU-M-P0}] \label{lem:saw} Let $\mu$ be a weight satisfying $A_{\infty}$ condition. If $u \in A_1(\mu)$ and $u v\in A_{\infty}(\mu)$ then 
	$$\left\|\frac{M_{\mu}f}{v}\right\|_{L^{1,\infty}(u v \mu)} \lesssim \left\| f \right\|_{L^{1}(u \mu)}.$$ 
\end{lemma}  
This result appeared first in \cite{CU-M-P0} in the particular case $\mu=1$, but using the $A_{\infty}$ condition the same proof there also works.

Now, we state the main result in this section. 

\begin{theorem}\label{thm:offdia-g} Let $\mu$ a weight satisfying the $A_{\infty}$ condition, and let $1\le r_0\le p_0<\infty$, $0<q_0< s_0$ and $\alpha\in (0,p_0]$. Assume that for certain measurable functions $f$ and $g$ and for all $w\in \widehat A_{\frac{p_0}{r_0}, \frac{\delta_{0}}{r_0}}(\mu)$, where $\delta_0$ is defined via 
$1/{\delta_0}= 1/{q_0}-1/{s_0}$, there holds 
\begin{equation*}
\|g\|_{L^{q_0,\infty}(w^{\frac {q_0}{r_0}} \mu^{\frac {q_0}{\delta_0}})}\le C \|f\|_{L^{p_0,\alpha}(w^{\frac{p_0}{r_0}}\mu^{1-\frac{p_0}{r_0}})}.
\end{equation*}
Then for $0<q<s_0$ with 
\[
\frac 1q-\frac 1{q_0}= \frac 1{r_0}-\frac 1{p_0}
\]
and all $w\in \widehat A_{1, \frac{\delta_1}{r_0}}(\mu)$, where $1/{\delta_1}= \frac{1}{q}-\frac{1}{s_0}$, we have 
\[
\| g\|_{L^{q,\infty}(w^{\frac q{r_0}} \mu^{\frac q{\delta_1}})}\le C \|f\|_{L^{r_0,\frac{\alpha r_0}{p_0}}(w)}.
\]
\end{theorem}

Before providing a proof of this result, we observe that in the particular case $r_0=1$ and $\alpha=1$ the previous theorem is contained in \cite[Theorem 8.10]{ro:ro}.

\begin{proof}
Let $H=|f|^{r_0}w^{1-\frac{\delta_1}{r_0}} \mu^{-1} $. Then set
\[
E=\Big\{  w^{\frac{\delta_1}{r_0}\cdot \frac{q}{s_0}} \mu^{\frac q{s_0}}M_\mu H >(\gamma y)^{r_0} \Big\},\quad F=\Big\{  w^{\frac{\delta_1}{r_0}\cdot \frac{q}{s_0}} \mu^{\frac q{s_0}} M_\mu H \le (\gamma y)^{r_0}, |g|>y \Big\}. 
\]
Then we have 
\begin{align*}
\int_{\{|g|>y\}}w^{\frac q{r_0}} \mu^{\frac q{\delta_1}}&\le \int_{E}w^{\frac q{r_0}} \mu^{\frac q{\delta_1}}
+\int_{F}w^{\frac q{r_0}} \mu^{\frac q{\delta_1}}.
\end{align*}
For the first term, we have 
\begin{align*}
\int_{E}w^{\frac q{r_0}} \mu^{\frac q{\delta_1}}= \int_E w^{\frac{\delta_1}{r_0}} \big(  w^{\frac{\delta_1}{r_0}\cdot \frac{q}{s_0}} \mu^{\frac q{s_0}}\big)^{-1} \mu. 
\end{align*}
Since $w^{\frac{\delta_1}{r_0}}\in A_1(\mu)$ and 
\[
 w^{\frac{\delta_1}{r_0}} \big(  w^{\frac{\delta_1}{r_0}\cdot \frac{q}{s_0}} \mu^{\frac q{s_0}}\big)^{-1} = 
\big( w^{\frac{\delta_1}{r_0}} \big)^{1- \frac q{s_0}} (\mu^{-1})^{\frac q{s_0}}\in A_\infty(\mu),
\]
by Lemma \ref{lem:saw} and $\alpha\le p_0$ we have 
\begin{align}\label{eq:estimoE}
	\int_{E}w^{\frac q{r_0}} \mu^{\frac q{\delta_1}}&\lesssim (\gamma y)^{-r_0}\int H w^{\frac{\delta_1}{r_0}}  \mu= (\gamma y)^{-r_0} \int |f|^{r_0} w\nonumber\\
	&\lesssim  (\gamma y)^{-r_0} \|f\|_{L^{r_0, \frac{\alpha r_0}{p_0}}(w)}^{r_0}.
\end{align}

It remains to estimate the integral over $F$. Let 
\[
\beta= \frac{q_0/{r_0}}{({p_0}/{r_0})'}.
\]
Then it is easy to check that $1+\beta= q_0/q$ and therefore
\[
1- \frac q{s_0}(1+\beta)= \frac{q_0}{\delta_0}.
\]Hence we have 
\begin{align*}
\int_{F}w^{\frac q{r_0}} \mu^{\frac q{\delta_1}}&\le (\gamma y)^{r_0\beta} \int_{\{|g|>y\}} \frac{w^{\frac q{r_0}} \mu^{\frac q{\delta_1}}}{ (M_\mu H)^\beta (w^{\frac{\delta_1}{r_0}\cdot \frac{q}{s_0}} \mu^{\frac q{s_0}}  )^\beta}\\
&=  (\gamma y)^{r_0\beta}  \int_{\{|g|>y\}}  w^{\frac{\delta_1}{r_0}(1-\frac q{s_0}(1+\beta))} \mu^{1-\frac q{s_0}(1+\beta) } (M_\mu H)^{-\beta}\\
&=  (\gamma y)^{r_0\beta} \int_{\{|g|>y\}}  \big[w^{\frac{\delta_1}{\delta_0}}  (M_\mu H)^{-\frac 1{(p_0/{r_0})'}} \big]^{\frac {q_0}{r_0}}\mu^{\frac{q_0}{\delta_0} }.
\end{align*}
To proceed we need 
\[
v:= w^{\frac{\delta_1}{\delta_0}}  (M_\mu H)^{-\frac 1{(p_0/{r_0})'}} \in \widehat A_{p_0/{r_0}, \delta_{0}/{r_0}}(\mu).
\]
This is indeed the case since 
\begin{align*}
\big[ w^{\frac{\delta_1}{\delta_0}}  (M_\mu H)^{-\frac 1{(p_0/{r_0})'}}\big]^{\delta_{0}/{r_0} }= w^{\delta_1/{r_0}}(M_\mu H)^{-\frac {\delta_{0}/{r_0} }{(p_0/{r_0})'}}\in \widehat A_{1+ \frac {\delta_{0}/{r_0} }{(p_0/{r_0})'}}(\mu).
\end{align*}
Thus by the assumption and $M_\mu H\ge H$ we have 
\begin{align}\label{eq:estimoF}
\int_{F}&w^{\frac q{r_0}} \mu^{\frac q{\delta_1}}\nonumber\\&\le (\gamma y)^{r_0\beta} \int_{\{|g|>y\}} v^{\frac {q_0}{r_0}}\mu^{\frac{q_0}{\delta_0} }\nonumber\\
&\lesssim  \gamma^{r_0\beta} y^{r_0\beta-q_0}\|f\|_{L^{p_0,\alpha}(v^{\frac{p_0}{r_0}}\mu^{1-\frac{p_0}{r_0}})}^{q_0}\nonumber\\
&\sim  \gamma^{r_0\beta} y^{r_0\beta-q_0} \Big(\int_0^\infty s^{\alpha-1} \Big(\int_{\{|f|>s\}} w^{\frac{\delta_1}{\delta_0}\cdot \frac{p_0}{r_0}} (M_\mu H)^{1-\frac{p_0}{r_0}}\mu^{1-\frac{p_0}{r_0}}\Big)^{\frac \alpha {p_0}}\ud s\Big)^{\frac{q_0}\alpha}\nonumber\\
&\le \gamma^{r_0\beta} y^{r_0\beta-q_0} \Big(\int_0^\infty s^{\alpha-1} \Big(\int_{\{|f|>s\}} w^{\frac{\delta_1}{\delta_0}\cdot \frac{p_0}{r_0}} H^{1-\frac{p_0}{r_0}}\mu^{1-\frac{p_0}{r_0}}\Big)^{\frac \alpha {p_0}}\ud s\Big)^{\frac{q_0}\alpha}\nonumber\\
&= \gamma^{r_0\beta} y^{r_0\beta-q_0} \Big(\int_0^\infty s^{\alpha-1} \Big(\int_{\{|f|>s\}} w  |f|^{r_0-p_0}\Big)^{\frac \alpha {p_0}}\ud s\Big)^{\frac{q_0}\alpha}\nonumber\\
&\le  \gamma^{r_0\beta} y^{r_0\beta-q_0} \Big(\int_0^\infty s^{\frac{\alpha r_0}{p_0}-1} \Big(\int_{\{|f|>s\}} w  \Big)^{\frac \alpha {p_0}}\ud s\Big)^{\frac{q_0}\alpha}\nonumber\\
&\sim  \gamma^{r_0\beta} y^{r_0\beta-q_0} \|f\|_{L^{r_0, \frac{\alpha r_0}{p_0}}(w)}^{\frac{q_0r_0}{p_0}}.
\end{align}
Combining estimates \eqref{eq:estimoE} and \eqref{eq:estimoF} we have 
\begin{align*}
\int_{\{|g|>y\}}w^{\frac q{r_0}} \mu^{\frac q{\delta_1}}\lesssim  (\gamma y)^{-r_0}\|f\|_{L^{r_0, \frac{\alpha r_0}{p_0}}(w)}^{r_0}+ \gamma^{r_0\beta} y^{r_0\beta-q_0} \|f\|_{L^{r_0, \frac{\alpha r_0}{p_0}}(w)}^{\frac{q_0r_0}{p_0}}.
\end{align*}
Optimizing the choice of $\gamma$ we get 
\[
\int_{\{|g|>y\}}w^{\frac q{r_0}} \mu^{\frac q{\delta_1}}\lesssim y^{-q} \|f\|_{L^{r_0, \frac{\alpha r_0}{p_0}}(w)}^q
\]
and the result follows. 
\end{proof}
	\section{The classes $A_{\vec{p}, \vec r}^{\mathcal{R}}$ and proof of Theorem \ref{thm:main-general}}\label{sec:auxiliar} 
We recall the definition of the class $A_{\vec p, \vec r}^{\mathcal R}$. We say $\vec w\in A_{\vec p, \vec r}^{\mathcal R}$, if 
\begin{align*}
[\vec w]_{A_{\vec p, \vec r}}^{\mathcal R}:= \sup_Q \Big(\frac 1{|Q|}\int_Q w^{\frac{\delta_{m+1}}{p}}\Big)^{\frac 1{\delta_{m+1}}}\prod_{i=1}^m \| \chi_Q w_i^{-\frac 1{p_i}}\|_{L^{\delta_i, \infty}(\frac{\ud x}{|Q|})}<\infty,
\end{align*}
where $w=\prod\limits_{i=1}^m w_i^{p/{p_i}}$, $1\le r_1,\ldots, r_{m+1}<\infty$, $r_i\le p_i$ for all $1\le i\le m$, $p< r_{m+1}'$ and 
\begin{align*}
\frac 1p=\sum_{i=1}^m\frac 1{p_i},\quad \frac 1{\delta_i}:= \frac 1{r_i}-\frac 1{p_i}\,\,\forall \,1\le i\le m+1,\quad \frac 1{p_{m+1}}=1-\frac 1p.
\end{align*}
For notational convenience, we also set $\frac{1}{r}=\sum\limits_{i=1}^{m+1}\frac{1}{r_i}$ and the following expressions
\[
\frac1\varrho:= \frac1{r_m}-\frac 1{r_{m+1}'}+\sum_{i=1}^{m-1}\frac 1{p_i},\qquad \mu:= \Big( \prod_{i=1}^{m-1} w_i^{\frac 1{p_i}}\Big)^{\varrho}.
\]

We now supply the sufficient understanding of the classes $ A_{\vec p, \vec r}^{\mathcal R}$ to prove Theorem \ref{thm:main-general}.

	\begin{proposition} \label{impli-g}
		\begin{enumerate}[\rm(a)] 
		\item Let $1\le r_i\le p_i<\infty$ for $1\le i \le m-1$ and $p_m=r_m$ and $\frac{1}{p}=\sum_i \frac{1}{p_i} > \frac{1}{r_{m+1}'}$. Then 
$\vec w\in A_{\vec p, \vec r}^{\mathcal R}$ if and only if $(w_1, \cdots, w_{m-1}, 1)\in A_{\vec p, \vec r}^{\mathcal R}$ and $w_m^{\varrho/{r_m}} \in A_1(\mu)$.
                \item  Let $1\le r_i\le p_i<\infty$ for $1\le i \le m$ and $\frac{1}{p}=\sum_i \frac{1}{p_i}> \frac{1}{r_{m+1}'}$.  Assume that  $(w_1, \cdots, w_{m-1}, 1)\in A_{(p_1,\cdots, p_{m-1},r_m), \vec r}^{\mathcal R}$ and $W^{\delta_{m+1}/{r_{m}}} \in \widehat A_{1+ \frac{\delta_{m+1}}{\delta_m}}(\mu)$. If we set $w_m= W^{\frac{p_m}{r_m}}\mu^{-\frac{p_m}{\delta_m}} $, then $(w_1, \cdots, w_m)\in A_{\vec p, \vec r}^{\mathcal R}$. 
	\end{enumerate}
	\end{proposition}

Assuming Proposition \ref{impli-g} we can give the proof of Theorem \ref{thm:main-general}.

\begin{proof}[Proof Theorem \ref{thm:main-general}] 
If necessary, we may reindex  the subscript so that we always have $p_m>r_m$ (if $p_i=r_i$ for all $1\le i\le m$ then there is nothing to prove). Thus by part (b) of Proposition \ref{impli-g} and the assumption we know that for every fixed 
$(w_1, \cdots, w_{m-1}, 1)\in A_{(p_1,\cdots, p_{m-1},r_m), \vec r}^{\mathcal R}$
and all $W\in \widehat A_{\frac{p_m}{r_m}, \frac{\delta_{m+1}}{{r_m}}}(\mu)$ (i.e. $W^{\frac{\delta_{m+1}}{r_m}} \in \widehat A_{1+\frac {\delta_{m+1}}{\delta_m}}(\mu)$),
\[
\|g\|_{L^{p,\infty}(W^{\frac p{r_m}} \mu^{\frac p{\delta_{m+1}}})}\lesssim \big(\prod_{i=1}^{m-1}\|f_i\|_{L^{p_i,\alpha_i}(w_i)} \big)\| f_m\|_{L^{p_m,\alpha_m}(W^{\frac{p_m}{r_m}}\mu^{1-\frac{p_m}{r_m}})}.
\]
Then applying Theorem \ref{thm:offdia-g} with 
\[
r_0=r_m,\quad p_0=p_m,\quad \delta_0=\delta_{m+1},\quad q_0=p,\quad s_0=r_{m+1}',\quad \delta_1=\varrho
\] for all $V\in \widehat A_{1, \varrho/r_m}(\mu)$ (in other words $V^{\varrho/r_m}\in A_1(\mu)$), we have 
\begin{equation}\label{eq:qm11}
\|g\|_{L^{q,\infty}(V^{\frac q{r_m}} \mu^{\frac q{\varrho}})}\lesssim \big(\prod_{i=1}^{m-1}\|f_i\|_{L^{p_i,\alpha_i}(w_i)} \big)\| f_m\|_{L^{r_m,\frac{\alpha_m r_m}{p_m}}(V)},
\end{equation}
where $q$ is defined via $1/q=1/{r_m}+\sum_{i=1}^{m-1}1/{p_i}$. 
Now, by part (a) of Proposition \ref{impli-g} we have that \eqref{eq:qm11} holds for all $$(w_1, \cdots, w_{m-1}, V)\in A_{(p_1, \cdots, p_{m-1}, r_m), \vec r}^{\mathcal R}.$$ Repeating the above process will send us to the desired estimate. 
\end{proof}

The rest of this section will be dedicated to proving Proposition \ref{impli-g}. 
 With the previous notation, first we show the regularity that the class $A_{\vec p, \vec r}^{\mathcal R}$ provides to the weight $w$.

\begin{lemma}\label{lem:doubmu}
Suppose that $\vec w\in A_{\vec p, \vec r}^{\mathcal R}$. Then $w^{\delta_{m+1}/p}\in A_{(\frac1r-1)\delta_{m+1}}^{\mathcal R}$. In particular, if $(w_1, \cdots, w_{m-1}, 1)\in A_{(p_1,\cdots, p_{m-1},r_m), \vec r}^{\mathcal R}$, then 
$\mu\in A_{(\frac1r-1)\varrho}^{\mathcal R}\subset A_\infty$. 
\end{lemma}
\begin{proof}
Note that
\[
\sum_{i=1}^m \frac 1{\delta_i}= \sum_{i=1}^m \Big(\frac 1{r_i}-\frac 1{p_i}\Big)=\frac 1r- \frac 1{r_{m+1}}-\frac 1p=\Big(\frac 1r-1\Big)-\frac 1{\delta_{m+1}}.
\]
Let $$
s_i= \delta_i \Big[\Big(\frac 1r-1\Big)-\frac 1{\delta_{m+1}} \Big].
$$
Then by H\"older's inequality for weak type Lebesgue spaces,
\begin{align*}
&\Big\| \chi_Q w^{-\frac{\delta_{m+1}}p\cdot  \frac 1{(\frac 1r-1)\delta_{m+1}}}\Big\|_{L^{((\frac1r-1)\delta_{m+1})',\infty }(\frac{\ud x}{|Q|})}\\
&\lesssim \prod_{i=1}^m \Big\|\chi_Q w_i^{- \frac{\delta_{m+1}}{p_i}\cdot  \frac 1{(\frac 1r-1)\delta_{m+1}}}\Big\|_{L^{s_i((\frac1r-1)\delta_{m+1})',\infty }(\frac{\ud x}{|Q|})}\\
&= \prod_{i=1}^m \Big\| \chi_Q w_i^{- \frac{\delta_{m+1}}{p_i}\cdot  \frac 1{(\frac 1r-1)\delta_{m+1}}\cdot (\frac 1r-1)}\Big\|_{L^{\delta_i,\infty }(\frac{\ud x}{|Q|})}^{\frac 1{\frac 1r-1}}=\prod_{i=1}^m \|   \chi_Q w_i^{-\frac 1{p_i}}\|_{L^{\delta_i,\infty }(\frac{\ud x}{|Q|})}^{\frac{r}{1-r}}.
\end{align*}
Thus 
\begin{align*}
&\Big\| \chi_Q w^{\frac{\delta_{m+1}}p\cdot  \frac 1{(\frac 1r-1)\delta_{m+1}}}\Big\|_{L^{(\frac1r-1)\delta_{m+1}}(\frac{\ud x}{|Q|})}\\
&\hspace{3cm}\times\Big\| \chi_Q w^{-\frac{\delta_{m+1}}p\cdot  \frac 1{(\frac 1r-1)\delta_{m+1}}}\Big\|_{L^{((\frac1r-1)\delta_{m+1})',\infty }(\frac{\ud x}{|Q|})}\\
&\lesssim \Big(\frac 1{|Q|}\int_Q w^{\frac{\delta_{m+1}}{p}}\Big)^{\frac 1{\delta_{m+1}}\cdot \frac r{1-r}}\prod_{i=1}^m \|   \chi_Q w_i^{-\frac 1{p_i}}\|_{L^{\delta_i,\infty }(\frac{\ud x}{|Q|})}^{\frac{r}{1-r}}\\
&\le \big( [\vec w]_{A_{\vec p, \vec r}}^{\mathcal R}\big)^{\frac r{1-r}}.
\end{align*}
\end{proof}

Now, we are ready to prove Proposition \ref{impli-g}.

\begin{proof}[Proof Proposition \ref{impli-g}]We start with (a).
	Suppose that $(w_1, \cdots, w_{m-1}, 1)\in A_{\vec p, \vec r}^{\mathcal R}$ and $w_m^{\varrho/{r_m}}\in A_1(\mu)$. Then by definition,
	\begin{align*}
	\sup_Q \Big(\frac 1{|Q|}\int_Q \mu \Big)^{\frac 1{\varrho}}\prod_{i=1}^{m-1} \| \chi_Q w_i^{-\frac 1{p_i}}\|_{L^{\delta_i, \infty}(\frac{\ud x}{|Q|})}<\infty
	\end{align*}
	and 
	\[
	\sup_Q \frac{\frac 1{\mu(Q)}\int_Q w_m^{\varrho/{r_m}}\mu}{\essinf_{Q} w_m^{\varrho/{r_m}} }<\infty.
	\]
	It then follows that 
	\begin{align*}
	&\sup_Q \Big(\frac 1{|Q|}\int_Q w_m^{\varrho/{r_m}}\mu\Big)^{\frac 1{\varrho}}\prod_{i=1}^{m-1} \| \chi_Q w_i^{-\frac 1{p_i}}\|_{L^{\delta_i, \infty}(\frac{\ud x}{|Q|})} \esssup_Q w_m^{-\frac 1{r_m}}\\
	&=\sup_Q\Big[\Big(\frac 1{|Q|}\int_Q \mu \Big)^{\frac 1{\varrho}}\prod_{i=1}^{m-1} \| \chi_Q w_i^{-\frac 1{p_i}}\|_{L^{\delta_i, \infty}(\frac{\ud x}{|Q|})} \Big]\cdot\Big[ \Big(  \frac{\frac 1{\mu(Q)}\int_Q w_m^{\varrho/{r_m}}\mu}{\essinf_{Q} w_m^{\varrho/{r_m}} }\Big)^{1/{\varrho}}\Big]\\
	&\le [(w_1, \cdots, w_{m-1},1)]_{A_{\vec p, \vec r}^{\mathcal R}} [w_m^{\varrho/{r_m}}]_{A_1(\mu)}^{1/{\varrho}}.
	\end{align*}
	On the contrary, suppose that $\vec w\in A_{\vec p, \vec r}^{\mathcal R}$. Then since 
	\[
	\Big(\frac 1{|Q|}\int_Q \mu \Big)^{\frac 1{\varrho}}\prod_{i=1}^{m-1} \| \chi_Q w_i^{-\frac 1{p_i}}\|_{L^{\delta_i, \infty}(\frac{\ud x}{|Q|})}\gtrsim 1
	\]
	and 
	\[
	\Big(  \frac{\frac 1{\mu(Q)}\int_Q w_m^{\varrho/{r_m}}\mu}{\essinf_{Q} w_m^{\varrho/{r_m}} }\Big)^{1/{\varrho}}\ge 1.
	\]
	By the same computation we get that $(w_1, \cdots, w_{m-1}, 1)\in A_{\vec p, \vec r}^{\mathcal R}$ and $w_m^{\varrho/{r_m}}\in A_1(\mu)$ immediately. 
	
	Next we prove part (b). We may without loss of generality assume that $p_m>r_m$ since otherwise it is already included in part (a). Then we have that $W^{\delta_{m+1}/{r_m}}=u_m (M_\mu g)^{-\frac {\delta_{m+1}}{\delta_m}}$ with some $u_m\in A_1(\mu)$. We will prove that 
\begin{equation}\label{eq:main1}
\| \chi_Q w_m^{-\frac 1{p_m}} \|_{L^{\delta_m,\infty}(\frac{\ud x}{|Q|})}\lesssim \frac{\essinf_Q (M_\mu g)^{\frac{1}{\delta_m}}}{\essinf_Q u_m^{\frac{1}{\delta_{m+1}}}}\Big(\frac {\mu(Q)}{|Q|}\Big)^{\frac{1}{\delta_m}}.
\end{equation}
In fact, if we prove \eqref{eq:main1}, then 
\begin{align*}
 &\Big(\frac 1{|Q|}\int_Q w^{\frac{\delta_{m+1}}{p}}\Big)^{\frac 1{\delta_{m+1}}}\prod_{i=1}^m \| \chi_Q w_i^{-\frac 1{p_i}}\|_{L^{\delta_i, \infty}(\frac{\ud x}{|Q|})}\\
 &\lesssim \Big(\frac 1{|Q|}\int_Q \mu W^{\frac{\delta_{m+1}}{r_m}}\Big)^{\frac 1{\delta_{m+1}}}\prod_{i=1}^{m-1} \| \chi_Q w_i^{-\frac 1{p_i}}\|_{L^{\delta_i, \infty}(\frac{\ud x}{|Q|})} \frac{\essinf\limits_Q (M_\mu g)^{\frac{1}{ \delta_m}}}{\essinf\limits_Q u_m^{\frac{1}{\delta_{m+1}}}}\Big(\frac {\mu(Q)}{|Q|}\Big)^{\frac{1}{\delta_m}}\\
 &\le  \Big(\frac 1{|Q|}\int_Q \mu u_m\Big)^{\frac 1{\delta_{m+1}}}\prod_{i=1}^{m-1} \| \chi_Q w_i^{-\frac 1{p_i}}\|_{L^{\delta_i, \infty}(\frac{\ud x}{|Q|})}\frac{1}{\essinf\limits_Q u_m^{\frac{1}{\delta_{m+1}}}}\Big(\frac {\mu(Q)}{|Q|}\Big)^{\frac{1}{\delta_m}}\\
 &\lesssim\Big(\frac {\mu(Q)}{|Q|}\Big)^{\frac{1}{\delta_m} +\frac 1{\delta_{m+1}}} \prod_{i=1}^{m-1} \| \chi_Q w_i^{-\frac 1{p_i}}\|_{L^{\delta_i, \infty}(\frac{\ud x}{|Q|})}\\
 &\le [(w_1, \cdots, w_{m-1},1)]_{A_{(p_1,\cdots, p_{m-1},r_m), \vec r}^{\mathcal R}}.
\end{align*}
Thus it remains to prove \eqref{eq:main1}. By definition,
\[
w_m^{-\frac 1{p_m}}= u_m^{-\frac 1{\delta_{m+1}}} (M_\mu g)^{\frac 1{\delta_m}}\mu^{\frac 1{\delta_m}}.
\]
Thus 
\begin{align*}
\| \chi_Q w_m^{-\frac 1{p_m}} \|_{L^{\delta_m,\infty}(\frac{\ud x}{|Q|})}&\lesssim\| \chi_Q \mu^{\frac 1{\delta_m}} (M_\mu g)^{\frac 1{\delta_m}} \|_{L^{\delta_m,\infty}(\frac{\ud x}{|Q|})}\frac{1}{\essinf\limits_Q u_m^{\frac{1}{\delta_{m+1}}}}\\
&= \| \chi_Q \mu M_\mu g \|_{L^{1,\infty}(\frac{\ud x}{|Q|})}^{\frac 1{\delta_m}}\frac{1}{\essinf\limits_Q u_m^{\frac{1}{\delta_{m+1}}}}.
\end{align*}
By Lemma \ref{lem:doubmu} we know that $\mu$ is doubling. It is well-known that 
\[
\chi_Q  M_\mu g \sim \chi_Q \essinf_Q M_\mu g + \chi_Q  M_\mu (g\chi_{3Q}).
\]
Hence by applying Lemma \ref{lem:saw} to $f=g\chi_{3Q}$ and $v=\mu^{-1}\in A_\infty(\mu)$ and $u=1$ we have 
\begin{align*}
\| \chi_Q \mu M_\mu g \|_{L^{1,\infty}(\frac{\ud x}{|Q|})}&\lesssim   \Big(\frac{\mu(Q)}{|Q|} \Big)\essinf_Q M_\mu g+|Q|^{-1} \| \chi_Q \mu M_\mu (g\chi_{3Q}) \|_{L^{1,\infty}}\\
&\lesssim \Big(\frac{\mu(Q)}{|Q|} \Big)\essinf_Q M_\mu g +  |Q|^{-1} \int_{3Q} |g| \mu \\
&\lesssim \Big(\frac{\mu(Q)}{|Q|} \Big)\essinf_Q M_\mu g, 
\end{align*}
where in the last step we used again that $\mu$ is doubling. The proof is now complete.
	\end{proof}

Now, to finish this section we prove the following result which, in particular, shows that in the case $r_i =1$ for every $i=1,\dots,m+1$ then $A_{\vec p,\vec r}^{\mathcal R} \equiv A_{\vec p}^{\mathcal R}$.

\begin{lemma}\label{red} 
	Let  $q\ge 1$ and $v$ be a weight. For fixed $0<a<1$ and $0\le b<1$, we have that
	\begin{equation}\label{eq:equivnorm}
		\left\|\chi_{Q} v^{-1} \right\|_{L^{q,\infty}(v)}\sim \left\|\chi_{Q} v^{-a} \right\|^{k}_{L^{kq,\infty}(v^{b})},
	\end{equation}
	where $k=\frac{1}{a q'}+\frac{b}{aq}.$ 
\end{lemma}
\begin{proof} Let $S= \sup_{t>0} t \, v (\{ x\in Q: t<v^{-1}(x)\le 2 t \})^{1/q}$. Since $1/q\le 1$  we have that  
	\begin{align*}
		\,t \, v(\{ x\in Q: t<v^{-1}(x) \})^{1/q} &\leq\sum_{k=0}^{\infty}  \,t \, v( \{ x\in Q: 2^{k} t<v^{-1}(x)\le 2^{k+1} t \})^{1/q}\\
		&\leq S  \sum_{k=0}^{\infty} 2^{-k}= 2 S.
	\end{align*}
	Then $\left\|\chi_{Q} v^{-1} \right\|_{L^{q,\infty}(v)}\sim S.$ Now with this estimate, equivalence \eqref{eq:equivnorm} follows directly. In fact, since $0<a<1$, we have that
	\begin{align*}
		\left\|\chi_{Q} v^{-1} \right\|_{L^{q,\infty}(v)}&\sim\sup_{t>0}tv(\{ x\in Q: v^{-1}(x)\sim t \})^{1/q} \\
		&\sim \sup_{t>0}t[t^{(b-1)}v^{b}(\{ x\in Q: v^{-a}(x)\sim t^{a} \})]^{1/q} \\
		&\sim\sup_{t>0}t^{\frac{1}{q'}+\frac{b}{q}}v^{b}(\{ x\in Q: v^{-a}(x)\sim t^{a} \})^{1/q} \\
		&\sim\sup_{t>0}\left(tv^{b}(\{ x\in Q: v^{-a}(x)\sim t \})^{1/kq} \right)^k
	\end{align*}
	where $k=\frac{1}{a q'}+\frac{b}{a q}.$
\end{proof}
\begin{remark}\label{rem:dfnequiv}
Observe that if we apply the previous lemma systematically with $q=p'_i$, $a=\frac{1}{p_i}$ $(i=1,\dots, m)$ and $b=0$ (so that $k=1$) we have
\[
 \| \chi_Q w_i^{-1} \|_{L^{p_i', \infty}\big(\frac{w_i}{|Q|}\big)}\sim \| \chi_Q w_i^{-\frac 1{p_i}} \|_{L^{p_i', \infty}\big(\frac{\ud x}{|Q|}\big)}.
\]
With this at hand it is obvious to see that if  $r_i =1$ for every $i=1,\dots,m+1$ then $A_{\vec p,\vec r}^{\mathcal R} \equiv A_{\vec p}^{\mathcal R}$.
\end{remark}

	\section{Applications}\label{sec:applications}

To show the first application, we need to introduce some notation. For brevity, we will present it in the bilinear case, though the same arguments can be extended to the multilinear case. Before stating the theorem, we will define an operator as bi-sublinear if it simultaneously satisfies the following conditions:
$$|T(f_1+f_2,g)|\leq |T(f_1,g)|+|T(f_2,g)| $$ and 
$$|T(f,g_1+g_2)|\leq |T(f,g_1)|+|T(f,g_2)|, $$
where we assume that $|T|$ is increasing in each coordinate, and $\supp f_1\cap \supp f_2=\emptyset$ and  $\supp g_1\cap \supp g_2=\emptyset$, respectively.   
\begin{theorem} \label{thm:app}
Let $T$ be a bi-sublinear operator. Assume that there exists $0<\alpha \le1$ such that for every measurable Lebesgue sets $E$ and $F$  in $\mathbb{R}^n$ (with finite measure) and $\lambda_1,\lambda_2>0$ we have that  
\begin{equation}\label{hypo}
|T(\lambda_1 \chi_E,\lambda_2 \chi_F)(x)|\lesssim (\lambda_1 \lambda_2 )^{\alpha}  \mathcal{M}(\chi_E,\chi_F)(x)^{\alpha},\quad a.e.\quad  x\in \R^n.
\end{equation}
Then 
\[
\| |T(f_1,f_2)|^{1/\alpha}\|_{L^{\frac{1}{2},\infty}(v_1^{1/2}v_2^{1/2})}\lesssim \prod_{i=1}^2 \|f_i\|_{L^{1,\frac{\alpha}{q}}(v_i)},
\]
for all $(v_1, v_2)\in A_{\vec{\frac{1}{2}}}$ and $q>2\alpha$.  
\end{theorem} 
 \begin{proof}
 Observe that even that  if $T$ is sublinear, in general $S(f_1,f_2)=|T(f_1,f_2)|^{1/\alpha}$ could not be for $\alpha<1$. We fix $q>2\alpha$ and we take $p=\frac{q}{2}$ and $p_1=p_2=q$. We claim that for every pair of weights $(w_1,w_2)\in A^{\mathcal{R}}_{\vec{p}}$ the operator $S$ satisfies that

  \begin{equation}\label{eq:S}
  \|S (f_1,f_2)\|_{L^{p,\infty}(w_1^{1/2} w_2^{1/2})}\lesssim \prod_{i=1}^2 \|f_i\|_{L^{p_i,\alpha}(w_i)}.
  \end{equation}
  Obtaining that estimate the conclusion will follow by applying Theorem \ref{thm:main-general}
  Fix non-negative functions $f_1$ and $f_2$. For $j, i \in \mathbb{Z}$  we will use the notation $E_{j}=\{x\in \mathbb{R}^n : 2^{j}\le f_1(x)<2^{j+1} \}$,  $F_{i}=\{x\in \mathbb{R}^n : 2^{i}\le f_2(x)<2^{i+1} \}$. 
  
 Then using the hypothesis of $T$ it is not difficult to check that  
 $$|T(f_1,f_2)|\lesssim \sum_{i,j\in \mathbb{Z}} 2^{\alpha j} 2^{\alpha i} \mathcal{M}(\chi_{E_j},\chi_{F_i})^{\alpha}.$$ 
 
 Then, using that $\frac{p}{ \alpha}>1$ 
  
 $$\|S (f_1,f_2)\|^{\alpha}_{L^{p,\infty}(w_1^{1/2} w_2^{1/2})}=\| |T(f_1,f_2)| \|_{L^{\frac{p}{\alpha},\infty}(w_1^{1/2} w_2^{1/2})}$$
 
 $$\lesssim \sum_{i,j\in \mathbb{Z}} 2^{\alpha j} 2^{\alpha i} \|\mathcal{M}(\chi_{E_j},\chi_{F_i})^{\alpha}\|_{L^{\frac{p}{\alpha},\infty}(w_1^{1/2} w_2^{1/2})}.$$
 
 But since $$\|\mathcal{M}(\chi_{E_j},\chi_{F_i})^{\alpha}\|_{L^{\frac{p}{\alpha},\infty}(w_1^{1/2} w_2^{1/2})}=\|\mathcal{M}(\chi_{E_j},\chi_{F_i})\|^{\alpha}_{L^{p,\infty}(w_1^{1/2} w_2^{1/2})},$$ and since $(w_1,w_2)\in A^{\mathcal{R}}_{\vec{p}}$ we can obtain that 
 
 $$\|S (f_1,f_2)\|^{\alpha}_{L^{p,\infty}(w_1^{1/2} w_2^{1/2})}\lesssim \sum_{i,j\in \mathbb{Z}} 2^{\alpha j} 2^{\alpha i} w_1(E_j)^{\frac{\alpha}{p_1}} w_2(F_i)^{\frac{\alpha}{p_2}}. $$
 
 From that using standard arguments it is not difficult to check that the previous sum is essentially bounded by $\left(\prod_{i=1}^2 \|f_i\|_{L^{p_i,\alpha}(w_i)}\right)^{\alpha}$ and claim (\ref{eq:S}) follows.
 
 \end{proof}

Now, we will apply the previous theorem to the following particular operator

$$N(f,g)(x)=\sup_{\lambda>0} \lambda |\{y\in \mathbb{R}^n \setminus \{0\}: \frac{|f(x+y)| |g(x+y)|}{|y|^{2n}}>\lambda\}|^{2} $$ 

We observe that the operator $N$ can be regarded as an specific instance, considering the additive group, of a bilinear variant of the (sublinear) maximal operators discussed in \cite{CLM-R}. In that work, intriguing applications to ergodic theory were derived employing an extension of the restricted extrapolation theorem within the linear setting, applied to weights belonging to one-sided Sawyer's classes. In fact using similar ideas as in \cite{CLM-R} (see also \cite{Assa} and \cite{C-DS}) we could deduce endpoint estimates for a discrete variant of $N$. However, for the sake of brevity and since we can already show the potential of Theorem \ref{thm:app} with the operator $N$, we remain in that particular continuous context.  We would like also mention that other bilinear variants (in the discrete setting)  with negative results already were considered in \cite{Assa-Bu}.  

\begin{corollary}\label{cor:app} Let $q>1$. Then 
\[
\| N(f_1,f_2)\|_{L^{\frac{1}{2},\infty}(v_1^{1/2}v_2^{1/2})}\lesssim \prod_{i=1}^2 \|f_i\|_{L^{1,\frac{1}{2q}}(v_i)},
\]
\end{corollary}

\begin{proof}
Observe that if we denote $$T(f_1,f_2)=N(f_1,f_2)^{1/2}.$$  It is not difficult to check that $T$ is a bi-sublinear operator. On the other hand for fixed measurable sets $E$ and $F$
 $$T(\lambda_1 \chi_E ,\lambda_2 \chi_F)(x)\lesssim \sup_{\lambda} \lambda \int_{|y|<\frac{(\lambda_1 \lambda_2)^{1/2n}}{\lambda^{1/n}}} \chi_{E}(x+y)\chi_F(x+y).$$

So, rescaling and using H\"older we can check that $T$ satisfies (\ref{hypo})  with $\alpha=\frac{1}{2}$. Therefore, $T$ verifyes  the hypothesis of Theorem \ref{thm:app}.

\end{proof}

Finally, we will describe another abstract result focussing on the natural fact that if $p>1$ then the space $L^{p,\infty}(\mu)$  is a normable space, while at the endpoint, only if some $p_i=1$ in the relationship $\frac{1}{p}=\frac{1}{p_1}+\dots +\frac{1}{p_m}$  will imply that $p<1$, and therefore arguments using duality cannot be directly applied. 

\begin{theorem} Assume that for each $j=1,\dots,$ we have an operator \newline $T_j (f_1,\dots,f_m)$  defined for functions $f_i \in L^{p_i}(w_i)$ for $i=1,\dots,m$  with  $\frac{1}{p}=\frac{1}{p_1}+\dots +\frac{1}{p_m}$. Such that for some $p>1$ and for every vector weight  $\vec w=(w_1, \cdots, w_m)\in A_{\vec p}^{\mathcal R}$ we have 
\begin{equation*}
\|T_j (f_1,\dots,f_m)\|_{L^{p,\infty}(\prod_{i=1}^m w_i^{p/{p_i}})}\leq c_{j}(\vec w) \prod_{i=1}^m \|f_i\|_{L^{p_i,1}(w_i)},
\end{equation*}   
where the non-negative scalars $c_{j}(\vec w)$ satisfy that  $\sum_{j=1}^{\infty} c_{j}(\vec w) \leq C(\vec w) <\infty$. 
Then, the operator $T=\sum_j T_j$ satisfies 
\[
\|T(f_1,\dots,f_m) \|_{L^{1/m,\infty}(\prod_{i=1}^m v_i^{1/{m}})}\lesssim \prod_{i=1}^m \|f_i\|_{L^{1,1/{p_i}}(v_i)},
\]
for all $(v_1, \cdots, v_m)\in A_{\vec{\frac{1}{m}}}$,   
\end{theorem}
The previous theorem is a direct consequence of Theorem \ref{thm:main} and the fact that for $p>1$,  $L^{p,\infty}(\prod_{i=1}^m w_i^{p/{p_i}})$ satisfies Minkowski's inequality. Of course, we could apply the previous theorem in the case that  we have only an operator $T$.

We finish this presentation with a last remark. 
\begin{remark}\label{rema}
The operator $N$ in Corollary \ref{cor:app} also shows the sharpness of Theorem \ref{thm:main-general}, in the sense that in general the right hand side of (\ref{eq:conclu}) it is not possible to change the $L^{1,1/p_i}(w_i)$ norm by $L^{1}(w_i)$ norm. In fact, assuming $n=1$, $N$ does not satisfy an estimate $L^1 \times L^1 \rightarrow L^{\frac{1}{2},\infty}$. If that is the case choosing $f=g$ we could deduce that the operator $N^*(f)=\sup_{\lambda>0} \lambda |\{y\in \mathbb{R} \setminus \{0\}: \frac{|f(x+y)|}{|y|}>\lambda\}|$ is weak type $(1,1)$. And it is well known that this is not the case, see for instance \cite{A-B-M}.
\end{remark}

\section*{Acknowledgement}

The three authors express profound gratitude to M. J. Carro. Her invaluable assistance, insightful comments, and engaging discussions were indispensable for the completion of this work. In particular, we are indebted to her for generously sharing with us a preliminary version of Theorem \ref{thm:offdia-g} and also the preprint \cite{cr:cr}, and for bringing  a variant of Remark \ref{rema} to our attention.

\end{document}